\documentclass[a4paper, 10pt, parskip=half]{scrartcl}

\usepackage[utf8]{inputenc}
\usepackage[T1]{fontenc}
\usepackage{lmodern}
\usepackage{csquotes}
\usepackage{amsmath}
\usepackage{amssymb}
\usepackage{amsthm}
\usepackage{amsfonts}
\usepackage{dsfont}
\usepackage{mathtools}
\usepackage{marginnote}
\usepackage[dvipsnames]{xcolor}
\usepackage{hyperref}
\hypersetup{%
  colorlinks=true,
  linkcolor=black,
  citecolor=black,
  urlcolor=blue,
  pdftitle={Upper estimates for the Hausdorff dimension of the temporal singular set in chemotaxis-fluid systems},
  pdfauthor={Mario Fuest},
  pdfkeywords={},
  bookmarksopen=true,
}

\usepackage{xcolor}

\usepackage[numbers, sort&compress]{natbib}

\newcommand{\R}{\mathbb{R}}

\newcommand{\N}{\mathbb{N}}

\newcommand{\mc}[1]{\mathcal{#1}}

\ifdefined\labelenumi%
  \renewcommand{\labelenumi}{(\roman{enumi})}
  
\fi

\newcommand{\eps}{\varepsilon}
\newcommand{\vt}{\vartheta}

\newcommand{\defs}{\coloneqq}

\newcommand{\sea}{\searrow}

\newcommand{\ol}{\overline}
\newcommand{\wt}{\widetilde}

\newcommand{\ddt}{\frac{\mathrm{d}}{\mathrm{d}t}}

\newcommand{\embed}{\hookrightarrow}

\newcommand{\hp}{\hphantom}
\newcommand{\pe}{\mathrel{\hp{=}}}

\newcommand{\intom}{\int_\Omega}

\newcommand{\intntom}{\int_0^T \int_\Omega}

\newcommand{\intninfom}{\int_0^\infty \int_\Omega}

\newcommand{\Ombar}{\ol \Omega}

\newcommand{\Ombarinf}{\Ombar \times [0, \infty)}
\newcommand{\Omegainf}{\Omega \times [0, \infty)}

\newcommand{\loc}{\mathrm{loc}}

\newcommand{\leb}[2][\Omega]{\ensuremath{L^{#2}(#1)}}
\newcommand{\lebl}[1][\Omega]{\ensuremath{L\log L(#1)}}
\newcommand{\lebs}[2][\Omega]{\ensuremath{L_\sigma^{#2}(#1)}}
\newcommand{\sob}[3][\Omega]{\ensuremath{W^{#2, #3}(#1)}}

\newcommand{\con}[2][\Ombar]{\ensuremath{C^{#2}(#1)}}

\newcommand{\neps}{n_\eps}
\newcommand{\net}{n_{\eps t}}
\newcommand{\nne}{n_{0 \eps}}
\newcommand{\ce}{c_\eps}
\newcommand{\cet}{c_{\eps t}}
\newcommand{\cne}{c_{0 \eps}}
\newcommand{\ue}{u_\eps}
\newcommand{\uet}{u_{\eps t}}
\newcommand{\une}{u_{0 \eps}}
\newcommand{\Pe}{P_\eps}
\newcommand{\Ye}{Y_\eps}
\newcommand{\Fe}{F_\eps}
\newcommand{\ye}{y_\eps}
\newcommand{\yej}{y_{\eps_j}}
\newcommand{\yoe}{y_{1 \eps}}
\newcommand{\yte}{y_{2 \eps}}

\newcommand{\tops}{\texorpdfstring}

\makeatletter
\renewenvironment{proof}[1][\proofname]{\par
  \pushQED{\qed}%
  \normalfont \topsep0\p@\relax
  \trivlist
  \item[\hskip\labelsep\scshape
  #1\@addpunct{.}]\ignorespaces
}{%
  \popQED\endtrivlist\@endpefalse
}
\makeatother

\newtheorem{base}{Base}[section]
\numberwithin{equation}{section}

\newtheorem{theorem}[base]{Theorem} \newtheorem*{theorem*}{Theorem}
\newtheorem{lemma}[base]{Lemma} \newtheorem*{lemma*}{Lemma}
 \newtheorem*{prop*}{Proposition}
 \newtheorem*{cor*}{Corollary}

\theoremstyle{definition}
\newtheorem{remark}[base]{Remark} \newtheorem*{remark*}{Remark}
\newtheorem{definition}[base]{Definition} \newtheorem*{definition*}{Definition}
 \newtheorem*{example*}{Example}
 \newtheorem*{cond*}{Condition}

\setkomafont{title}{\normalfont\Large}
\title{Upper estimates for the Hausdorff dimension of the temporal singular set in chemotaxis-fluid systems}
\usepackage{authblk}

\author{Mario Fuest\footnote{e-mail: fuest@ifam.uni-hannover.de}}
\affil{Leibniz Universität Hannover,\\ Institut für Angewandte Mathematik,\\ Welfengarten 1, 30167 Hannover, Germany}
\date{}

\begin{document}
\maketitle

\KOMAoptions{abstract=true}
\begin{abstract}
  \noindent
  The chemotaxis-fluid system
  \begin{align}\tag{$\star$}\label{prob:star}
    \begin{cases}
      n_t + u \cdot \nabla n = \Delta n - \nabla \cdot (n \nabla c), \\
      c_t + u \cdot \nabla c = \Delta c - nc, \\
      u_t + (u \cdot \nabla) u = \Delta u + \nabla P + n \nabla \Phi, \quad \nabla \cdot u = 0,
    \end{cases}
  \end{align}
  models aerobic bacteria interacting with a fluid via transportation and buoyancy.
  When posed on a three-dimensional, smoothly bounded, convex domain $\Omega$,
  \eqref{prob:star} complemented with suitable initial and boundary conditions is known to admit a global `weak energy solution',
  which recently has been shown to be smooth (after a redefinition on a set of measure $0$) in $\overline \Omega \times E$ for some countable union of open intervals $E$ with $|(0, \infty) \setminus E| = 0$.\\[2pt]
  The present paper investigates further regularity properties of this solution
  and proves that ($E$ can be chosen such that) the $\frac12$-dimensional Hausdorff measure of $(0, \infty) \setminus E$ vanishes and thus that in particular its Hausdorff dimension is at most $\frac12$.
  As $\frac12$ has been the best known upper estimate for the Hausdorff dimension of the temporal singular set for the unperturbed Navier--Stokes equations for quite some time,
  this result is the best one can hope for \eqref{prob:star} without significant progress in the regularity theory of (homogeneous) Navier--Stokes equations.
  \\[5pt]
 \textbf{AMS Classification (2020):} {35B65 (primary); 35K55, 35Q35, 35Q92, 92C17 (secondary)}
\end{abstract}

\section{Introduction}
While aerobic bacteria may strive to orient their movement towards oxygen-rich regions, in liquid environments their motion is also influenced by the velocity of a surrounding fluid, which in turn is affected by buoyancy forces caused by the bacteria's mass.
In order to capture these effects and in particular aiming to understand the spontaneous formation of the spatial patterns having been observed experimentally (\cite{DombrowskiEtAlSelfConcentrationLargeScaleCoherence2004}),
\cite{TuvalEtAlBacterialSwimmingOxygen2005} proposes the model
\begin{align}\label{prob:main}
  \begin{cases}
    n_t + u \cdot \nabla n = \Delta n - \nabla \cdot (n \nabla c)                            & \text{in $\Omega \times (0, \infty)$}, \\
    c_t + u \cdot \nabla c = \Delta c - nc                                                   & \text{in $\Omega \times (0, \infty)$}, \\
    u_t + (u \cdot \nabla) u = \Delta u + \nabla P + n \nabla \Phi, \quad \nabla \cdot u = 0 & \text{in $\Omega \times (0, \infty)$}, \\
    \partial_\nu n = \partial_\nu c = 0, \quad u = 0                                         & \text{on $\partial \Omega \times (0, \infty)$}, \\
    (n, c, u)(\cdot, 0) = (n_0, c_0, u_0)                                                    & \text{in $\Omega$},
  \end{cases}
\end{align}
where $\Phi$ is a given (gravitational) potential and $n_0, c_0, u_0$ are given suitably smooth initial data.
We consider \eqref{prob:main} in smooth, bounded, convex domains $\Omega \subset \R^3$.

The system consists of an intricate coupling
between a chemotaxis-consumption model as introduced by Keller and Segel (\cite{KellerSegelTravelingBandsChemotactic1971}) and the Navier--Stokes equations.
The regularity theory for the latter is infamously difficult in three-dimensional settings (\cite{LerayMouvementLiquideVisqueux1934}, \cite{WiegnerNavierStokesEquationsNeverending1999}, \cite{SohrNavierStokesEquations2001}),
while close relatives of the former are even known to exhibit finite-time blow-up phenomena
(see \cite{JagerLuckhausExplosionsSolutionsSystem1992}, \cite{HerreroVelazquezBlowupMechanismChemotaxis1997}, \cite{NagaiBlowupNonradialSolutions2001}, \cite{WinklerFinitetimeBlowupHigherdimensional2013} for a small selection of such results
and \cite{LankeitWinklerFacingLowRegularity2019} for a recent survey, for instance).

Accordingly, global classical existence results for \eqref{prob:main} so far require smallness of the data (\cite{DuanEtAlGlobalSolutionsCoupled2010}, \cite{CaoLankeitGlobalClassicalSmalldata2016})
or are limited to the two-dimensional setting (\cite{WinklerGlobalLargedataSolutions2012}).
(For the analysis of the large-time behavior in the latter case, we refer to \cite{WinklerStabilizationTwodimensionalChemotaxisNavier2014}, \cite{ZhangLiConvergenceRatesSolutions2015}.)
For three-dimensional domains $\Omega$,
global weak solutions have first been constructed for versions of \eqref{prob:main} without the nonlinear convection term $(u \cdot \nabla) u$ in the fluid equation (\cite{WinklerGlobalLargedataSolutions2012})
but later also for \eqref{prob:main} without any simplification (\cite{WinklerGlobalWeakSolutions2016}).
Moreover, not only these but all `eventual energy solutions' eventually become smooth and converge towards homogeneous equilibria (\cite{WinklerHowFarChemotaxisdriven2017}).

While the question whether these solutions are in fact classical remains elusive, one may still wonder which further partial regularity properties they enjoy.
Indeed, weak solutions to the unperturbed Navier--Stokes equations can only fail to be locally $L^\infty$ in space-time on a set with Hausdorff dimension at most $\frac53$ (\cite{SchefferNavierStokesEquationsBounded1980})
and for suitable solutions, i.e., weak solutions inter alia fulfilling a generalized energy inequality,
the upper estimate for the Hausdorff dimension of the spatio-temporal singular set can even be improved to $1$ (\cite{CaffarelliEtAlPartialRegularitySuitable1982}).
Very recently, a corresponding result for a close relative of \eqref{prob:main} posed on $\R^3$ has been proven (\cite{ChenEtAlHausdorffMeasureSingularity2023}).

In the present paper, we consider a different type of regularity,
namely the largeness of the set of times $t$ where a weak solution of \eqref{prob:main} is classical in $\Ombar \times U_t$ for some neighbourhood $U_t$ of $t$.
After all, the partial regularity results above do not exclude the possibility that there is an open interval $I$
such that  $u(t)$ does not belong to, say, $C^2(\Ombar; \R^3)$ for all $t \in I$ (even after a redefinition on a null set).
That this is not possible for weak solutions of the unperturbed Navier--Stokes equations satisfying the standard energy inequality
has already be shown by Leray in his celebrated structure theorem (\cite[Section~33]{LerayMouvementLiquideVisqueux1934})
which entails the stronger statement that the set of times $t$ for which such a neighbourhood $U_t$ does not exist, i.e., the \emph{temporal singular set} (TSS),
has Lebesgue measure $0$.
Going beyond that, it is moreover known that the $\frac12$-dimensional Hausdorff measure of the TSS is $0$ and that hence its Hausdorff dimension is at most $\frac12$,
see \cite[Proposition~4.1]{FoiasTemamAnalyticGeometricProperties1979} and \cite[Theorem~5]{GigaSolutionsSemilinearParabolic1986}.
(In fact, the crucial computations had already been performed by Leray, see \cite[Section~34]{LerayMouvementLiquideVisqueux1934}.)

In \cite{WinklerDoesLerayStructure2023}, Winkler has recently shown
that the TSSs of the weak solutions of \eqref{prob:main} constructed in \cite{WinklerGlobalWeakSolutions2016} have Lebesgue measure $0$.
Our main result expands on this and shows that also the more nuanced statement regarding the Hausdorff dimension of the TSS
carries over to the Navier--Stokes equations coupled via buoyancy and transportation to a chemotaxis-consumption system.
\begin{theorem}\label{th:main}
  Let $\Omega \subset \R^3$ be a smooth, bounded, convex domain and let
  \begin{allowdisplaybreaks}
  \begin{align}
    n_0 &\in \lebl \text{ be nonnegative with $n_0 \not\equiv 0$}, \label{eq:main:n_0}\\
    c_0 &\in \leb\infty \text{ be nonnegative with $\sqrt{c_0} \in \sob12$}, \label{eq:main:c_0} \\
    u_0 &\in \lebs2 \defs \{\,\varphi \in L^2(\Omega; \R^3) \mid \nabla \cdot \varphi = 0 \text{ in $\mc D'(\Omega)$}\,\}, \label{eq:main:u_0}\\
    \Phi &\in \con\infty \cap \sob2\infty. \notag
  \end{align}
  \end{allowdisplaybreaks}%
  Then there exist a global weak energy solution $(u, v, w)$ of \eqref{prob:main} in the sense of Definition~\ref{def:sol} below
  as well as $T_\star \in [0, \infty)$, a set $\mc I \subseteq \N$ and pairwise disjoint open intervals $I_\iota \subseteq (0, T_\star)$, $\iota \in \mc I$,
  such that the set
  \begin{align*}
    E \defs \bigcup_{\iota \in \mc I} I_\iota \cup (T_\star, \infty) 
  \end{align*}
  has the following properties:
  After a redefinition on a set of Lebesgue measure $0$,
  the triple $(n, c, u)$ belongs to $\big(C^{2, 1}(\Ombar \times E)\big)^5$,
  there is a function $P \in C^{1, 0}(\Ombar \times E)$ such that $(n, c, u, P)$ forms a classical solution of \eqref{prob:main} in $\Ombar \times E$
  and the $\frac12$-dimensional Hausdorff measure of $(0, \infty) \setminus E$ is $0$, implying that its Hausdorff dimension is at most $\frac12$.
\end{theorem}

\begin{remark}
  We note that Theorem~\ref{th:main} is only concerned with a specific global weak energy solution of \eqref{prob:main},
  namely the one constructed in \cite{WinklerGlobalWeakSolutions2016} as a limit of solutions to regularized systems.
  An interesting question, which we have to leave open here, is whether an analogous result holds for potential other solutions as well.
  One possible (but perhaps quite ambitious) approach consists of first extending the weak-strong uniqueness results available for the unperturbed Navier--Stokes equations
  (see \cite[Section~32]{LerayMouvementLiquideVisqueux1934}, \cite{SerrinInitialValueProblem1963} and \cite{SohrvonWahlSingularSetUniqueness1984}, for instance) to \eqref{prob:main}
  and then to directly work on the level of weak solutions.
\end{remark}

\paragraph{Main ideas.}
The proof in \cite{WinklerDoesLerayStructure2023}, which shows that the TSSs of certain weak solutions to \eqref{prob:main} have Lebesgue measure $0$,
is based on the functionals
\begin{align}\label{eq:intro:y_i}
  y_1(t) \defs \intom n^p(\cdot, t) + \intom |\nabla c(\cdot, t)|^{2p}, \quad
  y_2(t) \defs \intom |A^\frac{\alpha}{2} u(\cdot, t)|^2,
\end{align}
where $p \in (\frac32, 3)$ and $a \in (\frac12, 1)$ and where $A$ denotes the Stokes operator.
Formally, they fulfill $y_i \in L^{s_i}((0, T))$ for
\begin{align*}
  s_1 \defs s_1(p) \defs \frac{2}{3(p-1)}
  \quad \text{and} \quad
  s_2 \defs s_2(\alpha) \defs \frac1\alpha
\end{align*}
as well as $(y_1+y_2)' \le C (y_1+y_2+1)^{\vt}$ in $(0, T)$ for some $C > 0$ and $\vt > 1$.
The most crucial condition for the latter is $\vt \ge \max\{\vt_1, \vt_2\}$, where
\begin{align*}
  \vt_1 \defs \vt_1(p) \defs \frac{2p-1}{2p-3}
  \quad \text{and} \quad
  \vt_2 \defs \vt_2(\alpha) > \underline{\vt_2}(\alpha) \defs \frac{\alpha+\frac12}{\alpha-\frac12}.
\end{align*}
Here, $\vt_i$ corresponds to $y_i$:
If no transport and buoyancy terms were present, then one would obtain $y_i' \le C (y_i+1)^{\vt_i}$ in $(0, T)$.

Since boundedness of $y_1 + y_2$ can be used to bootstrap the boundedness of $(n, c, u)$ in stronger topologies (see the proof of Lemma~\ref{lm:y_bdd_conv_c2})
and due to a result linking the Hausdorff dimension of TSSs to the integrability of functionals solving superlinear ODIs
(see Lemma~\ref{lm:crit_hausdim} and Lemma~\ref{lm:ode_est}),
the estimates obtained in \cite{WinklerDoesLerayStructure2023} already indicate the following:
By choosing $p$ and $\alpha$ close to $3$ and $1$, respectively, noting that
\begin{align*}
  1 - \frac{s_1(3)}{\vt_1(3)-1} = \frac12
  \quad \text{and} \quad
  1 - \frac{s_2(1)}{\underline{\vt_2}(1)-1} = \frac12,
\end{align*}
and making the formal arguments above rigorous,
one can conclude that the Hausdorff dimensions of the TSSs of solutions
to both the fluid-free chemotaxis-consumption system and the unperturbed Navier--Stokes equations are at most $\frac12$.
Likewise, with this approach alone one can bound the Hausdorff dimension of the TSSs of solutions to the fully coupled system \eqref{prob:main} by
\begin{align}\label{eq:intro:hausdorff_first}
  1 - \frac{\min\{s_1(p), s_2(\alpha)\}}{\max\{\vt_1(p), \underline{\vt_2}(\alpha)\}-1}
\end{align}
which, however, is strictly larger than $\frac12$ for all $p \in (\frac32, 3]$ and all $\alpha \in (\frac12, 1]$.

The crucial new idea of the present paper, which will eventually allow us to improve on \eqref{eq:intro:hausdorff_first},
is to consider the functional
\begin{align}\label{eq:intro:y}
  y \defs y_1^\mu + y_2 + 1
\end{align}
instead of $y_1 + y_2$, where $\mu \defs \mu(p, \alpha) \defs \frac{s_1(p)}{s_2(\alpha)}$.
This functional, in particular the choice of $\mu$, has the benefit that the best known integrability condition for the summands $y_1^\mu$ and $y_2$ is the same;
they both belong to $L^{s_2}((0, T))$.
Moreover, neglecting the transport terms for a moment, the estimate $y_1' \le C (y_1+1)^{\vt_1}$ implies
\begin{align*}
  (y_1^\mu)' \le \tilde C (y_1^\mu+1)^{\vt_{1,\mu}} \quad \text{in $(0, T)$},
  \qquad \text{where} \qquad \vt_{1, \mu} \defs \vt_{1, \mu}(p, \alpha) \defs \frac{\vt_1(p)-1}{\mu(p, \alpha)}+1
\end{align*}
for some $\tilde C > 0$.
Since $\vt_{1, \mu}(3, 1) = 2 = \underline{\vt_2}(1)$ (with $\mu(3, 1) = \frac13$), $s_2(1) = 1$ and $1 - \frac{1}{2} = \frac12$,
there is indeed hope that this idea leads to the upper estimate $\frac12$ for the Hausdorff dimension of the TSS.

The cost of this approach, however, is of course that the terms stemming from the transport and buoyancy interactions
can no longer be dealt with as in \cite{WinklerDoesLerayStructure2023}.
Accordingly, the main part of our analysis focuses on suitably estimating the worrisome terms appearing when differentiating $y$ (Lemmata~\ref{lm:rhs_est_ct}--\ref{lm:buoyancy_est})
and thereby showing that $y$ solves a certain superlinear ODI (Lemma~\ref{lm:y_odi}).

Finally, let us note that in contrast to \cite{WinklerDoesLerayStructure2023} we can take the endpoints $p=3$ and $\alpha=1$ in \eqref{eq:intro:y_i},
the main reason being that \cite{WinklerDoesLerayStructure2023} makes it possible to work on time intervals where the approximate solutions converge in $C^{2, 1}$.
These choices not only simplify the computations below
but also ensure that the $\frac12$-dimensional Hausdorff measure (and not only the $d$-dimensional Hausdorff measure for $d > \frac12$) of the TSS is $0$.

\paragraph{Plan of the paper.}
The rest of the paper is organized as follows:
We recall global existence properties of approximate problems in Section~\ref{sec:approx},
we derive a superlinear ODI for the functional $\ye$ defined in \eqref{eq:intro:y} in Section~\ref{sec:func_ineq},
we take the limit of both the solutions and $\ye$ in Section~\ref{sec:eps_sea_0}
and make use of these preparations to show that the $\frac12$-dimensional Hausdorff measure of the TSS vanishes in Section~\ref{sec:hausdorff}.

\paragraph{Notation.}
Throughout the article, we fix a smooth, bounded, convex domain $\Omega \subset \R^3$, $n_0, c_0, u_0$ as in \eqref{eq:main:n_0}--\eqref{eq:main:u_0} and $\Phi \in \con\infty \cap \sob2\infty$.
Moreover, for a Lebesgue measurable set $E$ and $s \in (0, 1)$,
we denote by $\leb[E]{s}$ the set of (equivalence classes of) all real-valued measurable functions $\varphi$ on $E$
for which $\|\varphi\|_{\leb[E]s} \defs (\int_E |\varphi|^s)^\frac1s$ is finite. 

\section{Global solutions to approximate problems}\label{sec:approx}
We start by recalling a global existence result for solutions to certain approximate problems.
\begin{lemma}\label{lm:approx}
  Let $(\nne, \cne, \une)_{\eps \in (0, 1)} \subset C^\infty(\Ombar) \times C^\infty(\Ombar) \times \{\,\varphi \in C_c^\infty(\Omega) \mid \nabla \cdot \varphi = 0\,\}$ with
  \begin{align*}
    &\nne \ge 0 \text{ in $\Omega$}, \quad \|\nne\|_{\leb1} = \|n_0\|_{\leb1}, \\
    &\cne \ge 0 \text{ in $\Omega$}, \quad \sqrt{\cne} \in C^\infty(\Ombar), \quad \|\cne\|_{\leb\infty} \le \|c_0\|_{\leb\infty}, \\
    &\|\une\|_{\leb2} = \|u_0\|_{\leb2}
  \end{align*}
  for all $\eps \in (0, 1)$, and
  \begin{alignat*}{2}
    \nne &\to n_0 &&\qquad \text{in $\lebl$}, \\
    \sqrt{\cne} &\to \sqrt{c_0} &&\qquad \text{a.e.\ in $\Omega$ and in $\sob12$}, \\
    \une &\to u_0 &&\qquad \text{in $\leb2$}
  \end{alignat*}
  as $\eps \sea 0$.
  Furthermore, we denote by $A$ the Stokes operator on $\lebs2$ and by
  \begin{align*}
    \Ye v \defs (1 + \eps A)^{-1} v
    \qquad \text{for $v \in \lebs2$ and $\eps \in (0, 1)$},
  \end{align*}
  the Yosida approximation,
  and set
  \begin{align*}
    \Fe(s) \defs \frac1\eps \ln(1 + \eps s)
    \qquad \text{for $s \ge 0$ and $\eps \in (0, 1)$}.
  \end{align*}
  For each $\eps \in (0, 1)$, there exists a global classical solution $(\neps, \ce, \ue, \Pe) \in (C^{\infty}(\Ombarinf))^5 \times C^{1, 0}(\Omega \times (0, \infty))$ of
  \begin{align}\label{prob:approx}
    \begin{cases}
      \net + \ue \cdot \nabla \neps = \Delta \neps - \nabla \cdot (\neps \Fe'(\neps) \nabla \ce)                  & \text{in $\Omega \times (0, \infty)$}, \\
      \cet + \ue \cdot \nabla \ce = \Delta \ce - F(\neps) \ce                                                     & \text{in $\Omega \times (0, \infty)$}, \\
      \uet + (\Ye \ue \cdot \nabla) \ue = \Delta \ue + \nabla \Pe + \neps \nabla \Phi, \quad \nabla \cdot \ue = 0 & \text{in $\Omega \times (0, \infty)$}, \\
      \partial_\nu \neps = \partial_\nu \ce = 0, \quad \ue = 0                                                    & \text{on $\partial \Omega \times (0, \infty)$}, \\
      (\neps, \ce, \ue)(\cdot, 0) = (\nne, \cne, \une)                                                            & \text{in $\Omega$}
    \end{cases}
  \end{align}
  such that $\neps, \ce \ge 0$ in $\Ombarinf$.
\end{lemma}
\begin{proof}
  This is essentially contained in \cite[Lemma~2.2 and Lemma~3.9]{WinklerGlobalWeakSolutions2016},
  which, however, only assert that $(\neps, \ce, \ue) \in (C^{2, 1}(\Ombarinf))^5$.
  The claimed $C^\infty$ regularity follows from \cite[Theorem~IV.5.3]{LadyzenskajaEtAlLinearQuasilinearEquations1988} and \cite[Section~V.1.8 and Section~V.2.3]{SohrNavierStokesEquations2001}.
\end{proof}
For each $\eps \in (0, 1)$, we henceforth fix $\nne, \cne, \une, \Ye, \Fe, \neps, \ce, \ue$ as in Lemma~\ref{lm:approx}.

We also recall two yet very basic properties of these solutions.
\begin{lemma}
  For all $\eps \in (0, 1)$,
  \begin{align}\label{eq:approx:mass_con}
    \intom \neps(\cdot, t) = \intom n_0
    \qquad \text{for all $t > 0$}
  \end{align}
  and
  \begin{align}\label{eq:approx:c_bdd}
    \|\ce(\cdot, t)\|_{\leb\infty} \le \|c_0\|_{\leb\infty}
    \qquad \text{for all $t > 0$}.
  \end{align}
\end{lemma}
\begin{proof}
  While \eqref{eq:approx:mass_con} follows from integrating the first equation in \eqref{prob:approx},
  \eqref{eq:approx:c_bdd} is a consequence of the maximum principle applied to the second equation in \eqref{prob:approx}.
\end{proof}

\section{A new functional inequality}\label{sec:func_ineq}
As already stated in the introduction, a key new ingredient in our proof consists of the analysis of the functionals
\begin{align}\label{eq:def_y}
  \ye(t) \defs \yoe^\mu(t) + \yte(t) + 1, \quad
  \yoe(t) \defs \intom \neps^3(\cdot, t) + \intom |\nabla \ce(\cdot, t)|^{6}, \quad
  \yte(t) \defs \intom |\nabla \ue(\cdot, t)|^2
\end{align}
for $t \in (0, \infty)$ and $\eps \in (0, 1)$.
Herein, $\mu > 0$ is a parameter whose influence we track for the sake of exposition;
finally, we will choose $\mu = \frac13$.

As a first step towards estimating the functional $\ye$, we note the following.
\begin{lemma}\label{lm:y_first_est}
  Let $\mu > 0$.
  Then there exists $C > 0$ such that
  \begin{align}\label{eq:y':ineq}
    &\pe  \ye'(t)
          + \frac{\yoe^{\mu-1}}{C} \left( \intom |\nabla \neps^\frac{3}{2}|^2 + \intom \big|\nabla |\nabla \ce|^3\big|^2 \right)
          + \frac1C \intom |\Delta \ue|^2 \notag \\
    &\le  C \yoe^{\mu-1} \left( \intom \neps^3 |\nabla \ce|^2 + \intom \neps^2 |\nabla \ce|^{4} + \intom |\nabla \ce|^{6} \cdot |\nabla \ue| \right) \notag \\
    &\pe  + C \left( \left| \intom \Delta \ue \cdot \mc P\Bigl( (\Ye \ue \cdot \nabla) \ue \Bigr)\right| + \left| \intom \Delta \ue \cdot \mc P\Bigl( \neps \nabla \Phi \Bigr) \right| \right)
  \end{align}
  in $(0, \infty)$ for all $\eps \in (0, 1)$, where $\mc P$ denotes the Helmholtz projection.
\end{lemma}
\begin{proof}
  As in \cite[Lemma~3.1]{WinklerDoesLerayStructure2023},
  performing two testing procedures,
  while making use of solenoidality of $\ue$,
  Young's inequality,
  convexity of $\Omega$ (which implies $\partial_\nu |\nabla \ce|^2 \le 0$ on $\partial \Omega \times (0, \infty)$, see \cite[Lemme~2.I.1]{LionsResolutionProblemesElliptiques1980})
  and \eqref{eq:approx:c_bdd},
  shows
  \begin{align}\label{eq:y':ineq:n}
        \frac13 \ddt \intom \neps^3
        + \frac49 \intom |\nabla \neps^\frac{3}{2}|^2 
    \le \intom \neps^3 |\nabla \ce|^2
  \end{align}
  as well as
  \begin{align}\label{eq:y':ineq:c}
        \frac{1}{6} \ddt\intom |\nabla \ce|^{6}
        + \frac{4}{9} \intom \big|\nabla |\nabla \ce|^3 \big|^2
    \le \frac{\big(4 + \sqrt3\big)^2 \|c_0\|_{\leb\infty}^2}{4} \intom \neps^2 |\nabla \ce|^{4} + \intom |\nabla \ce|^{6} \cdot |\nabla \ue|
  \end{align}
  in $(0, \infty)$ for all $\eps \in (0, 1)$.
  Moreover, testing the (projected) fluid equation with $-\Delta \ue$ gives
  \begin{align}\label{eq:y':ineq:u}
        \frac12 \ddt \intom |\nabla \ue|^2
        + \intom |\Delta \ue|^2
    =   - \intom \Delta \ue \cdot \mc P\Bigl( (\Ye \ue \cdot \nabla) \ue \Bigr)
        + \intom \Delta \ue \cdot \mc P\Bigl( \neps \nabla \Phi \Bigr)
  \end{align}
  in $(0, \infty)$ for all $\eps \in (0, 1)$.
  Since $\yoe > 0$ by \eqref{eq:main:n_0} and \eqref{eq:approx:mass_con}, $\ye$ is differentiable with
  \begin{align*}
      \ye'
    = \mu \yoe^{\mu-1} \yoe' + \yte'
    = \mu \yoe^{\mu-1} \ddt \left( \intom \neps^3 + \intom |\nabla \ce|^{6} \right) + \ddt \left( \intom |\nabla \ue|^2 \right)
  \end{align*}
  in $(0, \infty)$ for all $\eps \in (0, 1)$.
  Combined with \eqref{eq:y':ineq:n}--\eqref{eq:y':ineq:u}, this results in \eqref{eq:y':ineq} upon an evident choice of $C$.
\end{proof}

We shall now estimate the right-hand side in \eqref{eq:y':ineq} against the dissipative terms therein as well as powers of $\ye$, ultimately deriving a superlinear ODI solved by $\ye$.
We begin by treating the fluid-free terms stemming from $\yoe$.
As they share the factor $\yoe^{\mu-1}$ with the dissipative term contained in $\yoe'$ (because neither contains a contribution of $\yte$),
we can control them as in \cite[Lemma~3.2]{WinklerDoesLerayStructure2023} which deals with the special case $\mu=1$.
\begin{lemma}\label{lm:rhs_est_ct}
  Let $\mu > 0$.
  For each $\eta > 0$ we can find $C > 0$ such that
  \begin{align*}
        \yoe^{\mu-1} \left( \intom \neps^3 |\nabla \ce|^2 + \intom \neps^2 |\nabla \ce|^{4}\right)
    \le \eta \yoe^{\mu-1} \left( \intom |\nabla \neps^\frac 32|^2 + \intom \big|\nabla |\nabla \ce|^3\big|^2 \right)
        + C \ye^\frac{3\mu+2}{3\mu}
  \end{align*}
  in $(0, \infty)$ for all $\eps \in (0, 1)$.
\end{lemma}
\begin{proof}
  Following \cite[Lemma~3.2]{WinklerDoesLerayStructure2023}, we can find $c_1 > 0$ such that
  \begin{align*}
        \left( \intom \neps^3 |\nabla \ce|^2 + \intom \neps^2 |\nabla \ce|^{4}\right)
    \le \eta \intom |\nabla \neps^\frac 32|^2 + \eta \intom \big|\nabla |\nabla \ce|^3\big|^2 + c_1 \yoe^\frac{5}{3} + c_1
  \end{align*}
  in $(0, \infty)$ for all $\eps \in (0, 1)$.
  Since $\yoe^{\mu-1+\frac{5}{3}} \le \ye^{\frac{3\mu+2}{3\mu}}$ and $\ye \ge 1$, this entails the statement.
\end{proof}

While there is also a term in \eqref{eq:y':ineq} neither containing $(\neps, \ce)$ nor the factor $\yoe^{\mu-1}$,
we cannot treat it exactly as in \cite[Lemma~3.5]{WinklerDoesLerayStructure2023} since a functional including $\intom |A^\frac{\alpha}{2} \ue|^2$ for $\alpha < 1$ is considered there
and the case $\alpha = 1$ turns out to be slightly different. For the latter, we can argue as follows, however.
\begin{lemma}\label{lm:rhs_est_fluid}
  Let $\mu > 0$.
  For each $\eta > 0$ we can find $C > 0$ such that
  \begin{align}\label{eq:rhs_est_fluid:statement}
          \left| \intom \Delta \ue \cdot \mc P\Bigl( (\Ye \ue \cdot \nabla) \ue \Bigr)\right|
    &\le  \eta \intom |\Delta \ue|^2 + C \ye^3
  \end{align}
  in $(0, \infty)$ for all $\eps \in (0, 1)$.
\end{lemma}
\begin{proof}
  By means of Young's and Hölder's inequalities and since $\mc P$ is an orthogonal projection on $L^2(\Omega)$, we have
  \begin{align}\label{eq:rhs_est_fluid:1}
          \left| \intom \Delta \ue \cdot \mc P\Bigl( (\Ye \ue \cdot \nabla) \ue \Bigr)\right|
    &\le  \frac{\eta}{2} \|\Delta \ue\|_{\leb2}^2 + \frac{1}{2\eta} \|(\Ye \ue \cdot \nabla) \ue\|_{\leb2}^2 \notag \\
    &\le  \frac{\eta}{2} \|\Delta \ue\|_{\leb2}^2 + \frac{1}{2\eta} \|\Ye \ue\|_{\leb6}^2 \|\nabla \ue\|_{\leb3}^2
    \qquad \text{in $(0, \infty)$}.
  \end{align}
  According to a Sobolev embedding and the Gagliardo--Nirenberg inequality, there are $c_1, c_2 > 0$ with
  \begin{alignat*}{3}
    \|\varphi\|_{\leb6} &\le c_1 \|\nabla \varphi\|_{\leb2}
    &&\qquad \text{for all $\varphi \in W_0^{1, 2}(\Omega; \R^3)$ and} \\
    \|\nabla \varphi\|_{\leb3} &\le c_2 \|\Delta \varphi\|_{\leb2}^\frac12 \|\nabla \varphi\|_{\leb2}^\frac12
    &&\qquad \text{for all $\varphi \in W_0^{1, 2}(\Omega; \R^3)$},
  \end{alignat*}
  so that due to Young's inequality
  \begin{align}\label{eq:rhs_est_fluid:2}
          \frac{1}{2\eta} \|\Ye \ue\|_{\leb6}^2 \|\nabla \ue\|_{\leb3}^2
    &\le  \frac{c_1^2 c_2^2}{2\eta} \|\nabla \Ye \ue\|_{\leb2}^2 \|\nabla \ue\|_{\leb2} \|\Delta \ue\|_{\leb2} \\
    &\le  \frac{\eta}{2} \intom |\Delta \ue|^2 + \frac{c_1^4 c_2^4}{8\eta^3} \left( \intom |\nabla \Ye \ue|^2 \right)^2 \left( \intom |\nabla \ue|^2 \right)
  \end{align}
  in $(0, \infty)$ for all $\eps \in (0, 1)$.
  Denoting the Stokes operator on $\lebs2$ again by $A$
  and recalling that the Yosida approximation $\Ye$ both commutes with $A^\frac12$ on $\mc D(A^\frac12)$ and is nonexpansive on $\lebs2$, we moreover obtain
  \begin{align}\label{eq:rhs_est_fluid:3}
        \|\nabla \Ye \ue\|_{\leb2}
    =   \|A^{\frac12} \Ye \ue\|_{\leb2}
    =   \|\Ye A^{\frac12} \ue\|_{\leb2}
    \le \|A^{\frac12} \ue\|_{\leb2}
    =   \|\nabla \ue\|_{\leb2}
  \end{align}
  in $(0, \infty)$ for all $\eps \in (0, 1)$.
  In combination, \eqref{eq:rhs_est_fluid:1}--\eqref{eq:rhs_est_fluid:3} and \eqref{eq:def_y} yield \eqref{eq:rhs_est_fluid:statement} for $C \defs \frac{c_1^4 c_2^4}{8\eta^3}$.
\end{proof}

The remaining terms on the right-hand side in \eqref{eq:y':ineq} require new estimates as the corresponding dissipative terms now contain different coefficients whenever $\mu \neq 1$.
\begin{lemma}
  Let $\mu > 0$.
  For each $\eta > 0$ we can find $C > 0$ such that
  \begin{align*}
        \yoe^{\mu-1} \intom |\nabla \ce|^{6} \cdot |\nabla \ue| 
    \le \eta \yoe^{\mu-1} \intom \big|\nabla |\nabla \ce|^3\big|^2
        + C \ye^3
  \end{align*}  
  in $(0, \infty)$ for all $\eps \in (0, 1)$.
\end{lemma}
\begin{proof}
  The Gagliardo--Nirenberg inequality asserts that there is $c_1 > 0$ such that
  \begin{align*}
    \|\varphi\|_{\leb4}^2 \le c_1 \|\nabla \varphi\|_{\leb2}^\frac32 \|\varphi\|_{\leb2}^\frac12 + c_1 \|\varphi\|_{\leb2}^2
    \qquad \text{for all $\varphi \in \sob12$}.
  \end{align*}
  Recalling \eqref{eq:def_y}, this implies
  \begin{align*}
        \Big\||\nabla \ce|^6\Big\|_{\leb2}
    =   \Big\||\nabla \ce|^3\Big\|_{\leb4}^2
    \le c_1 \Big\|\nabla |\nabla \ce|^3\Big\|_{\leb2}^{\frac32} \yoe^\frac14 + c_1 \yoe
  \end{align*}
  in $(0, \infty)$ for all $\eps \in (0, 1)$,
  whenceupon we infer together with Hölder's and Young's inequalities and the estimates $\yoe \le \ye^\frac{1}{\mu}$ and $\yte \le \ye$ that
  \begin{align*}
          \yoe^{\mu-1} \intom |\nabla \ce|^{6} \cdot |\nabla \ue| 
    &\le  c_1 \yoe^{\mu-1} \left( \Big\|\nabla |\nabla \ce|^3\Big\|_{\leb2}^{\frac32} \yoe^\frac14 + \yoe \right) \|\nabla \ue\|_{\leb2} \\
    &=    c_1 \yoe^{\frac34(\mu-1)} \Big\|\nabla |\nabla \ce|^3\Big\|_{\leb2}^{\frac32} \yoe^{\frac{\mu}{4}} \yte^\frac12
          + c_1 \yoe^{\mu} \yte^\frac{1}{2} \\
    &\le  \eta \yoe^{\mu-1} \intom \big|\nabla |\nabla \ce|^3\big|^2
          + c_2 \ye^{4 (\frac{1}{4} + \frac{1}{2})}
          + c_2 \ye^{1+\frac12}
  \end{align*}
  in $(0, \infty)$ for all $\eps \in (0, 1)$.
  Since $4 (\frac{1}{4} + \frac{1}{2}) = 3$, $1+\frac12 < 3$ and $\ye \ge 1$, this entails the statement.
\end{proof}

Finally, for the last remaining term in \eqref{eq:y':ineq}
it turns out that it is beneficial to \emph{not} make use of the dissipative term $\yoe^{\mu-1} \intom |\nabla \neps|^2$ (at least not when $\mu$ is small)
but instead directly estimate the terms involving $\neps$ by $\yoe$.
(On the other hand, \cite[Lemma~3.6]{WinklerDoesLerayStructure2023} turns out to be stronger for $\mu=1$.)
\begin{lemma}\label{lm:buoyancy_est}
  Let $\mu > 0$.
  For each $\eta > 0$ we can find $C > 0$ such that
  \begin{align}\label{eq:buoyancy_est:statement}
          \left| \intom \Delta \ue \cdot \mc P\Bigl( \neps \nabla \Phi \Bigr) \right|
    &\le  \eta \intom |\Delta \ue|^2 + C \ye^{\frac{1}{2\mu}}
  \end{align}
  in $(0, \infty)$ for all $\eps \in (0, 1)$.
\end{lemma}
\begin{proof}
  By Young's inequality and the orthogonal projection property of $\mc P$, there is $c_1 > 0$ such that
  \begin{align}\label{eq:buoyancy_est:1}
          \left| \intom \Delta \ue \cdot \mc P\Bigl( \neps \nabla \Phi \Bigr) \right| 
    &\le  \eta \intom |\Delta \ue|^2 + c_1 \|\neps\|_{\leb2}^2
  \end{align}
  in $(0, \infty)$ for all $\eps \in (0, 1)$.
  Littlewood's inequality and \eqref{eq:approx:mass_con} then assert that with $\theta=\frac34$ and some $c_2 > 0$ we have
  \begin{align}\label{eq:buoyancy_est:2}
        \|\neps\|_{\leb2}^2
    \le \|\neps\|_{\leb3}^{2\theta} \|\neps\|_{\leb 1}^{2(1-\theta)}
    \le c_2 \|\neps\|_{\leb 3}^\frac{3}{2}
    \le c_2 \ye^\frac{1}{2\mu}
  \end{align}
  in $(0, \infty)$ for all $\eps \in (0, 1)$.
  Inserting \eqref{eq:buoyancy_est:2} into \eqref{eq:buoyancy_est:1} gives \eqref{eq:buoyancy_est:statement} for $C \defs c_1 c_2$.
\end{proof}

Combining Lemma~\ref{lm:y_first_est}--Lemma~\ref{lm:buoyancy_est} shows that $\ye$ solves a superlinear ODI.
\begin{lemma}\label{lm:y_odi}
  Let $\mu \ge \frac13$.
  Then there exists $K > 0$ such that
  \begin{align}\label{eq:y_odi:odi}
    \ye' \le K \ye^3
    \qquad \text{in $(0, \infty)$ for all $\eps \in (0, 1)$}.
  \end{align}
\end{lemma}
\begin{proof}
  Lemma~\ref{lm:y_first_est}--Lemma~\ref{lm:buoyancy_est} assert $\ye' \le K \ye^\vt$ in $(0, \infty)$ for
  \begin{equation*}\label{eq:y_le_2y:def_vt}
    \vt \defs \max\left\{\frac{3\mu+2}{3\mu},\, 3,\, \frac{1}{2\mu}\right\} = 3
  \end{equation*}
  some $K > 0$ and all $\eps \in (0, 1)$.
\end{proof}

\section{Limits of the solution and the functional \tops{$y_\eps$}{y\_eps}}\label{sec:eps_sea_0}
We recall the solution concept introduced in \cite{WinklerHowFarChemotaxisdriven2017} and \cite{WinklerDoesLerayStructure2023}.
\begin{definition}\label{def:sol}
  A triple of functions $(n, c, u)$ with
  \begin{align*}
    n &\in L_{\loc}^4(\Ombarinf) \cap L_{\loc}^2([0, \infty); \sob12) \text{ being nonnegative with } n^\frac12 \in L_{\loc}^2([0, \infty); \sob12), \\
    c &\in L_{\loc}^\infty(\Ombarinf) \cap L_{\loc}^4([0, \infty); \sob14) \text{ being nonnegative with } c^\frac14 \in L_{\loc}^4([0, \infty); \sob14) \text{ and} \\
    u &\in L_{\loc}^\infty([0, \infty); \lebs2 \cap L_{\loc}^2([0, \infty); W_0^{1, 2}(\Omega; \R^3))
  \end{align*}
  is called a \emph{global weak energy solution} of \eqref{prob:main} if
  \begin{align*}
      - \intninfom n \varphi_t
      - \intom n_0 \varphi(\cdot, 0)
    = - \intninfom \nabla n \cdot \nabla \varphi
      + \intninfom n \nabla c \cdot \nabla \varphi
      + \intninfom n u \cdot \nabla \varphi
  \end{align*}
  for all $\varphi \in C_c^\infty(\Ombarinf)$,
  \begin{align*}
      - \intninfom c \varphi_t
      - \intom c_0 \varphi(\cdot, 0)
    = - \intninfom \nabla c \cdot \nabla \varphi
      - \intninfom n c \varphi
      + \intninfom c u \cdot \nabla \varphi
  \end{align*}
  for all $\varphi \in C_c^\infty(\Ombarinf)$ and
  \begin{align*}
      - \intninfom u \cdot \varphi_t
      - \intom u_0 \cdot \varphi(\cdot, 0)
    = - \intninfom \nabla u \cdot \nabla \varphi
      + \intninfom u \otimes u \cdot \nabla \varphi
      + \intninfom n \nabla \Phi \cdot \varphi
  \end{align*}
  for all $\varphi \in C_c^\infty(\Omegainf; \R^3)$ with $\nabla \cdot \varphi \equiv 0$, if additionally
  \begin{align}\label{eq:sol:energy1}
          \frac12 \intom |u(\cdot, t)|^2
        + \int_{t_0}^t \intom |\nabla u|^2 
    \le \frac12 \intom |u(\cdot, t_0)|^2
        + \int_{t_0}^t \intom nu \cdot \nabla \Phi
    \qquad \text{for a.e.\ $t_0 > 0$ and all $t > t_0$},
  \end{align}
  and if finally there exist $\kappa > 0$ and $K > 0$ such that
  \begin{align}\label{eq:sol:energy2}
        \ddt \left\{
          \intom n \ln n
          + \frac12 \intom \frac{|\nabla c|^2}{c}
          + \kappa \intom |u|^2
        \right\}
        + \frac1K \left\{
          \intom \frac{|\nabla n|^2}{n}
          + \intom \frac{|\nabla c|^4}{c^3}
          + \intom |\nabla u|^2
        \right\}
    \le K
    \qquad \text{in $\mc D'((0, \infty))$}.
  \end{align}
\end{definition}

The following lemma summarizes key findings of \cite{WinklerHowFarChemotaxisdriven2017} and \cite{WinklerDoesLerayStructure2023}.
\begin{lemma}\label{lm:eps_sea_0}
  Let $(\neps, \ce, \ue)_{\eps \in (0, 1)}$ be as given by Lemma~\ref{lm:approx}.
  Then there exist a null sequence $(\eps_j)_{j \in \N} \subset (0, 1)$
  and a global weak energy solution $(n, c, u)$ of \eqref{prob:main} in the sense of Definition~\ref{def:sol} such that
  \begin{align}\label{eq:eps_sea_0:pw}
    (\neps, \ce, \ue) \to (n, c, u)
    \qquad \text{a.e.\ in $\Omega \times (0, \infty)$ as $\eps = \eps_j \sea 0$}.
  \end{align}
  Moreover, there exist $T_\star \in [0, \infty)$ and disjoint open intervals $I_\iota \subset (0, T_\star)$, $\iota \in \mc I \subseteq \N$, such that
  \begin{align}\label{eq:eps_sea_0:def_e}
    E_0 \defs \bigcup_{\iota \in I} I_\iota \cup (T_\star, \infty) 
  \end{align}
  has the properties that $|(0, \infty) \setminus E_0| = 0$,
  that (possibly after a redefinition on a Lebesgue null set)
  $(n, c, u) \in C^{2, 1}(\Ombar \times E_0; \R^5)$ together with some $P \in C^{0, 1}(\Ombar \times E_0)$ solves \eqref{prob:main} classically in $\Ombar \times E_0$
  and that finally
  \begin{align}\label{eq:eps_sea_0:e0_in_e}
    &\forall t_0 \in E_0 \; \exists \eta \in (0, t_0) : (\neps, \ce, \ue) \to (n', c', u') \text{ in } C^{2, 1}(\Ombar \times [t_0-\eta, t_0+\eta]; \R^5)\notag \\
    &\text{along some null sequence $\eps \sea 0$, where $(n', c', u') = (n, c, u)$ a.e.}
  \end{align}
\end{lemma}
\begin{proof}
  Convergence towards a weak energy solutions which eventually becomes smooth has been shown in \cite[Section~9 and Theorem~1.3]{WinklerHowFarChemotaxisdriven2017}.
  The properties of the set $E_0$ follow from \cite[Lemma~7.1, Lemma~7.2]{WinklerDoesLerayStructure2023}. 
\end{proof}

If $\ye$ remains bounded, the solutions not only converges a.e.\ but also in $C^{2, 1}$.
\begin{lemma}\label{lm:y_bdd_conv_c2}
  Let $\ye$ be as in \eqref{eq:def_y} and $(\eps_j)_{j \in \N}$ as well as $(n, c, u)$ be as given by Lemma~\ref{lm:eps_sea_0}.
  Let moreover $I \subseteq (0, \infty)$ be an open interval and suppose that there are $C > 0$ and $j_0 \in \N$ such that
  \begin{align*}
    y_{\eps_j} \le C \qquad \text{in $I$ for all $j \ge j_0$}.
  \end{align*}
  For any compact interval $K \subset I$, there then exists a triple $(n', c', u')$ with $(n', c', u') = (n, c, u)$ a.e.\ and
  \begin{align*}
    (n_{\eps_j}, c_{\eps_j}, u_{\eps_j}) \to (n', c', u') 
    \qquad \text{in $C^{2, 1}(\Ombar \times K)$ as $\eps = \eps_j \sea 0$}.
  \end{align*}
\end{lemma}
\begin{proof}
  This can be shown as in \cite[Section~4]{WinklerDoesLerayStructure2023}; we sketch the main steps.
  By definition of $\ye$ and due to the embedding $W_0^{1, 2}(\Omega; \R^3) \embed L^6(\Omega; \R^3)$,
  both $(\nabla c_{\eps_j})_{j \ge j_0}$ and $(u_{\eps_j})_{j \ge j_0}$ are bounded in $L^\infty(I; L^6(\Omega; \R^3))$,
  so that \eqref{eq:approx:mass_con} and semigroup estimates yield boundedness of $(n_{\eps_j})_{j \ge j_0}$ in $L^\infty(\Omega \times \wt K)$ for all compact intervals $\wt K$ with $K \subset (\wt K)^\circ \subset \wt K \subset I$.
  Then parabolic regularity theory (together with a cutoff argument) first yields Hölder bounds for $n_{\eps_j}$ and $u_{\eps_j}$
  and then $C^{2+\gamma, 1+\frac{\gamma}{2}}(\Ombar \times K; \R^5)$ bounds for some $\gamma \in (0, 1)$ for $u_{\eps_j}$, $n_{\eps_j}$ and $c_{\eps_j}$.
  
  If there existed $\eta_0 > 0$ and a subsequence $(j_k)_{k \in \N}$ of $(j)_{j \in \N}$ such that
  \begin{align}\label{eq:y_bdd_conv_c2:contra}
    \|(n_{\eps_{j_k}}, c_{\eps_{j_k}}, u_{\eps_{j_k}}) - (n', c', u')\|_{C^{2, 1}(\Ombar \times K; \R^5)} > \eta_0
  \end{align}
  for all $k \in \N$ and all $(n', c', u')$ equaling $(n, c, u)$ a.e.,
  then the $C^{2+\gamma, 1+\frac{\gamma}{2}}$ bounds would allow us to extract a further subsequence along which the solutions converge in $C^{2, 1}(\Ombar \times K; \R^5)$ to some $(n', c', u')$.
  Thanks to \eqref{eq:eps_sea_0:pw}, we have $(n', c', u') = (n, c, u)$ a.e.\ and hence \eqref{eq:y_bdd_conv_c2:contra} cannot hold.
\end{proof}

In addition, we shall need integrability properties of $(n, c, u)$ going beyond those required by Lemma~\ref{def:sol}.
Similarly to \eqref{eq:sol:energy1} and \eqref{eq:sol:energy2},
the following lemma rests on a quasi-energy structure for the approximate problem \eqref{prob:approx} observed in \cite{WinklerGlobalWeakSolutions2016} and a limit process.
\begin{lemma}\label{lm:yl1}
  Let $(n, c, u)$ be as given by Lemma~\ref{lm:eps_sea_0} and $T > 0$.
  Then the function 
  \begin{align}\label{eq:def_y_limit}
    y \colon (0, T) \to \R, \quad
    t \mapsto \left( \intom n^3(\cdot, t) + \intom |\nabla c(\cdot, t)|^{6} \right)^\frac13 + \intom |\nabla u(\cdot, t)|^2 + 1
  \end{align}
  belongs to $L^1((0, T))$.
\end{lemma}
\begin{proof}
  Based on a quasi-energy property of the functional $\intom \neps \ln \neps + \frac12 \intom |\nabla \ce^\frac12|^2 + K \intom |\ue|^2$ for appropriate $K > 0$,
  \cite[Lemma~3.6]{WinklerGlobalWeakSolutions2016} proves that there is $c_1 > 0$ such that
  \begin{align}\label{eq:yl1:spacetime}
    \intntom |\nabla \neps^\frac12|^2 + \intntom \frac{|D^2 \ce|^2}{\ce} + \intntom |\nabla \ue|^2 \le c_1
    \qquad \text{for all $\eps \in (0, 1)$.}
  \end{align}
  Together with \eqref{eq:approx:mass_con}, \eqref{eq:approx:c_bdd} and the embedding $\sob12 \embed \leb6$,
  this implies that both $(\neps^\frac12)_{\eps \in (0, 1)}$ and (each component of) $(\nabla \ce)_{\eps \in (0, 1)}$ are bounded in $X_1 \defs L^2((0, T); \leb6)$,
  while \eqref{eq:yl1:spacetime} directly entails boundedness of $(\nabla \ue)_{\eps \in (0, 1)}$ in $X_2 \defs L^2(\Omega \times (0, T); \R^{3 \times 3})$.
  Thus, there exists a subsequence of the sequence $(\eps_j)_{j \in \N}$ given by Lemma~\ref{lm:eps_sea_0}
  along which $(\neps^\frac12, \nabla \ce, \ue)$ converges weakly in $X_1^4 \times X_2$ to some limit $(\psi_1, \psi_2, \psi_3)$.
  By \eqref{eq:eps_sea_0:pw}, we may first conclude $(n^\frac12, \nabla c, u) = (\psi_1, \psi_2, \psi_3) \in X_1^4 \times X_2$ and then $y \in L^1((0, T))$.
\end{proof}

\section{Estimating the Hausdorff dimension}\label{sec:hausdorff}
In this section, we shall obtain upper estimates for the Hausdorff dimension of the temporal singular set of the solution given by Lemma~\ref{lm:eps_sea_0},
thereby proving Theorem~\ref{th:main}.
To that end, we rely on a criterion proven in \cite{KhaiTriHausdorffDimensionSingular2015}.
\begin{lemma}\label{lm:crit_hausdim}
  Let $T > 0$, let $E \subseteq (0, T)$ be open and suppose that $|(0, T) \setminus E| = 0$.
  If there are $C > 0$, $a > s > 0$ and a function $z \in L^s((0, T))$ with
  \begin{align*}
    z(t) \ge 0
    \quad \text{and} \quad
    \left(t, \min\left\{t + \frac{C}{z^a(t)}, T\right\}\right) \subseteq E
    \qquad \text{for all $t \in E$},
  \end{align*}
  then the $(1-\frac sa)$-dimensional Hausdorff measure of $S$ is $0$, i.e., $\mc H^{1-\frac sa}(S) = 0$.
\end{lemma}
\begin{proof}
  See \cite[Lemma~4]{KhaiTriHausdorffDimensionSingular2015}.
\end{proof}

The following elementary lemma shows how superlinear ODIs such as \eqref{eq:y_odi:odi} connect to the key condition in Lemma~\ref{lm:crit_hausdim}.
\begin{lemma}\label{lm:ode_est}
  Let $K > 0$ and $\sigma > 1$.
  Then there exists $C > 0$ such that given any open interval $I \subseteq (0, \infty)$ all nonnegative solutions $z \in C^1(I)$ of the ODI $z'(t) \le K z^\sigma(t)$, $t \in I$, fulfill
  \begin{align*}
    z(t) \le 2 z(t_0) \qquad \text{for all $t \in (t_0, t_0 + C z^{1-\sigma}(t_0)) \cap I$ and all $t_0 \in I$.}
  \end{align*}
\end{lemma}
\begin{proof}
  The comparison principle for ODEs warrants that for all $t_0 \in I$, 
  \begin{align*}
    z(t)
    \le \Bigl( z^{1-\sigma}(t_0) - K(\sigma-1)(t-t_0) \Bigr)^\frac{1}{1-\sigma}
    \qquad \text{for all $t \in (t_0, t_0 + \tfrac{z^{1-\sigma}(t_0)}{K(\sigma-1)}) \cap I$.}
  \end{align*}
  Setting $C \defs \frac{1-2^{1-\sigma}}{K(\sigma-1)} \in (0, \frac{1}{K(\sigma-1)})$,
  we obtain for all $t_0 \in I$ and all $t \in (t_0, t_0+C z^{1-\sigma}(t_0)) \cap I$ that
  \begin{align*}
        z(t)
    \le \left( z^{1-\sigma}(t_0) - K(\sigma-1) C z^{1-\sigma}(t_0) \right)^{\frac{1}{1-\sigma}}
    =   \left( (1 - (1-2^{1-\sigma})) z^{1-\sigma}(t_0) \right)^{\frac{1}{1-\sigma}}
    =   2 z(t_0),
  \end{align*}
  as desired.
\end{proof}

With these preparations at hand, we may now prove the essential part of Theorem~\ref{th:main}.
\begin{lemma}\label{lm:hausdorff_12}
  Let $(n, c, u)$ be as given by Lemma~\ref{lm:eps_sea_0}. 
  Then there exists an open set $E \subseteq (0, \infty)$ with the following properties:
  After a redefinition on a set of Lebesgue measure $0$, $(n, c, u)$ belongs to $C^{2, 1}(\Ombar \times E; \R^5)$,
  there is a function $P \in C^{1, 0}(\Ombar \times E)$ such that $(n, c, u, P)$ forms a classical solution of \eqref{prob:main} in $\Ombar \times E$
  and the $\frac12$-dimensional Hausdorff measure of $(0, \infty) \setminus E$ is $0$.
\end{lemma}
\begin{proof}
  Let $(\ye)_{\eps \in (0, 1)}$, $y$ and $(\eps_j)_{j \in \N}$ be as in \eqref{eq:def_y} (with $\mu = \frac13$), \eqref{eq:def_y_limit} and Lemma~\ref{lm:eps_sea_0}, respectively,
  and let $T > 0$.
  We set $X_\eta(t_0) \defs C^{2, 1}(\Ombar \times [t_0-\eta, t_0+\eta]; \R^5)$ for $\eta \in (0, \min\{t_0,T-t_0\})$ and
  \begin{align*}
    E_T \defs
    \Big\{\,t_0 \in (0, T) \;\Big\vert\;
      &\exists \eta \in (0, \min\{t_0, T-t_0\} : (\neps, \ce, \ue) \to (n', c', u') \text{ in } X_\eta(t_0) \\
      &\text{along some null sequence $\eps \sea 0$, where $(n', c', u') = (n, c, u)$ a.e.} \Big\}.
  \end{align*}
  Obviously, $E_T$ is open. 
  By Lemma~\ref{lm:eps_sea_0}, the set $E_0$ defined in \eqref{eq:eps_sea_0:def_e} satisfies $|(0, T) \setminus E_0| = 0$
  and $E_0 \cap (0, T)$ is contained in $E_T$, whence $|(0, T) \setminus E_T| = 0$.

  According to Lemma~\ref{lm:y_odi} and Lemma~\ref{lm:ode_est}, there is $c_1 > 0$ such that
  \begin{align}\label{eq:hausdorff_12:ye_bdd}
    \ye(t) \le 2 \ye(t_0)
    \qquad \text{for all $t \in \big(t_0, t_0 + 4c_1 \ye^{-2}(t_0)\big)$, all $t_0 \in (0, \infty)$ and all $\eps \in (0, 1)$.}
  \end{align}
  To show that
  \begin{align}\label{eq:hausdorff_12:U_T}
    U_T(t_0) \defs \big(t_0, \min\{t_0 + c_1 y^{-2}(t_0), T\}\big) \subseteq E_T
    \qquad \text{for all $t_0 \in E_T$},
  \end{align}
  we let $t_0 \in E_T$ and $t_1 \in U_T(t_0)$.
  The inclusion $t_0 \in E_T$ implies $\ye(t_0) \to y(t_0)$ and due to $y(t_0) \ge 1$ also $\frac{\ye(t_0)}{y(t_0)} \to 1$ as $\eps = \eps_j \sea 0$.
  That is, there is $j_0 \in \N$ such that $\yej(t_0) \le 2y(t_0)$ for all $j \ge j_0$.
  Since then $y^{-2}(t_0) \le 4 \yej^{-2}(t_0)$ for all $j \ge j_0$, it follows from \eqref{eq:hausdorff_12:ye_bdd} that
  \begin{align*}
        \yej(t)
    \le 2 \yej(t_0)
   \le 4 y(t_0)
    <   \infty
    \qquad \text{for all $t \in U_T(t_0)$ and all $j \ge j_0$.}
  \end{align*}
  Choosing $\eta > 0$ so small that also $t_1 - \eta$ and $t_1 + \eta$ belong to $U_T(t_0)$,
  we thus infer from Lemma~\ref{lm:y_bdd_conv_c2} that $(\neps, \ce, \ue) \to (n', c', u')$ in $X_\eta(t_1)$ as $\eps = \eps_j \sea 0$, where $(n', c', u') = (n, c, u)$ a.e.
  Hence, $t_1 \in E_T$; that is, \eqref{eq:hausdorff_12:U_T} holds.

  Since moreover $y \in L^1((0, T))$ by Lemma~\ref{lm:yl1},
  Lemma~\ref{lm:crit_hausdim} shows that the $\frac12$-dimensional Hausdorff measure $\mc H^\frac12$ of $(0, T) \setminus E_T$ vanishes.
  Setting $E \defs \bigcup_{T \in \N} E_T$, we apply countable subadditivity to see that also $\mc H^{\frac12}((0, \infty) \setminus E) = 0$.

  Given $t_0, \hat t_0 \in E$, there exist $\eta, \hat \eta > 0$ and $(n', c', u')$, $(\hat n', \hat c', \hat u')$ such that (along certain null sequences)
  $(\neps, \ce, \ue) \to (n', c', u')$ in $X_\eta(t_0)$, $(\neps, \ce, \ue) \to (\hat n', \hat c', \hat u')$ in $X_{\hat \eta}(\hat t_0)$ and 
  $(n', c', u') = (n, c, u) = (\hat n', \hat c', \hat u')$ a.e.
  If $M \defs [t_0-\eta, t_0+\eta] \cap [\hat t_0 - \hat \eta, \hat t_0 + \hat \eta] \neq \emptyset$, then $(n', c', u') = (\hat n', \hat c', \hat u')$ in $\Ombar \times M$.
  In particular, it is possible to redefine $(n, c, u)$ on a set of Lebesgue measure $0$ such that for all $t_0 \in E$,
  there exists $\eta > 0$ with $(\neps, \ce, \ue) \to (n, c, u)$ in $X_\eta(t_0)$ along some sequence $\eps \sea 0$.
  This regularity implies that the weak solution $(n, c, u)$ is actually also a classical solution in $\Ombar \times E$.
\end{proof}

\begin{proof}[Proof of Theorem~\ref{th:main}]
  We let $(n, c, u)$ be as given by Lemma~\ref{lm:eps_sea_0} and take $E$ as the union of the set $E$ given by Lemma~\ref{lm:hausdorff_12} and $(T_\star, \infty)$, where $T_\star$ is given by Lemma~\ref{lm:eps_sea_0}.
  Since open subsets of $(0, \infty)$ can always be written as a at most countable unions of open intervals,
  all claims are contained in Lemma~\ref{lm:eps_sea_0} and Lemma~\ref{lm:hausdorff_12}.
\end{proof}

\end{document}